\documentclass{amsart}%
\usepackage{amsfonts}
\usepackage{amsmath}
\usepackage{amssymb}
\usepackage{graphicx}%
\providecommand{\U}[1]{\protect\rule{.1in}{.1in}}

\theoremstyle{plain}

\newtheorem{proposition}{Proposition}
\theoremstyle{remark}
\newtheorem{remark}{Remark}
\begin{document}
\title[Dual instability measures]{Dual instability measures\ of a subspace of $P^{n}(K)$ under a subgroup of
$\operatorname{Aut}(K)$}
\author{Jun-ichi Matsushita}
\address{Tokyo, Japan}
\email{matsu-j2@outlook.jp}
\date{}
\subjclass[2020]{}

\begin{abstract}
Let $K$ be a commutative field and let $V$ be a subspace of $P^{n}(K)$. Let
$\Gamma$ be a subgroup of $\operatorname{Aut}(K)$ and let $\Gamma$ act on
$P^{n}(K)$ by $\sigma((x_{i})_{0\leq i\leq n})=(\sigma(x_{i}))_{0\leq i\leq
n}$ for $\sigma\in\Gamma$ and $(x_{i})_{0\leq i\leq n}\in P^{n}(K)$. In this
paper, we ask `how much'\ unstable $V$ is under $\Gamma$ by asking how much
higher (or lower) dimension the join (or the meet) of $\sigma(V)$ ($\sigma
\in\Gamma$) has than $V$, and answer it in terms of the Pl\"{u}cker
coordinates of $V$ and the invariant field $k$ of $\Gamma$, through presenting
dual `irrationality' measures of $V$ over $k$.

\end{abstract}
\maketitle

\section{\label{intr}Introduction}

Let $K$ be a commutative field and let $V$ be a subspace of $P^{n}(K)$ ($=$
the standard projective space of dimension $n$ over $K$). Let $\Gamma$ be a
subgroup of $\operatorname{Aut}(K)$ ($=$ the automorphism group of $K$) and
let $\Gamma$ act on $P^{n}(K)$ by
\begin{equation}
\sigma((x_{i})_{0\leq i\leq n})=(\sigma(x_{i}))_{0\leq i\leq n} \label{action}%
\end{equation}
for all $\sigma\in\Gamma$ and $(x_{i})_{0\leq i\leq n}\in P^{n}(K)$. Then
$\sigma(V)$ ($\sigma\in\Gamma$) are subspaces of $P^{n}(K)$ of the same
dimension as $V$, and, since $V$ is stable under $\Gamma$ if and only if the
join (or the meet) of $\sigma(V)$ ($\sigma\in\Gamma$)\ coincides with $V$, and
hence if and only if the join (or the meet) of $\sigma(V)$ ($\sigma\in\Gamma
$)\ has the same dimension as $V$, it is natural to ask `how much'\ unstable
$V$ is under $\Gamma$ by asking the following (i) or (ii):
\begin{align*}
&  \text{(i) How much higher dimension does the join of }\sigma(V)\text{
(}\sigma\in\Gamma\text{) have than }V\text{?}\\
&  \text{(ii) How much lower dimension does the meet of }\sigma(V)\text{
(}\sigma\in\Gamma\text{) have than }V\text{?}%
\end{align*}

In this paper, we answer these questions in terms of the Pl\"{u}cker
coordinates of $V$ and the invariant field of $\Gamma$, which is hereafter
denoted by $k$. For each $m$-dimensional subspace $X$ of $P^{n}(K)$, let
$\left(  \ldots,X_{j_{0}\cdots j_{m}},\ldots\right)  $\ denote the Pl\"{u}cker
coordinates of $X$ and define the $k$-\emph{irrationality degree}
$\operatorname{Irr}_{k}$ and the $k$-\emph{irrationality codegree}
$\operatorname{Irr}_{k}^{\ast}$ of $X$ by taking a permutation $j_{0}\cdots
j_{n}$ of $0\cdots n$ such that $X_{j_{0}\cdots j_{m}}\neq0$ and setting
\begin{align}
\operatorname{Irr}_{k}X  &  =\dim_{k}\left[  X\right]  _{j_{0}\cdots j_{n}%
}-\dim_{k}\left(  \left[  X\right]  _{j_{0}\cdots j_{n}}\cap k^{m+1}\right)
\text{,}\label{irrq}\\
\operatorname{Irr}_{k}^{\ast}X  &  =\dim_{k}\left[  X\right]  _{j_{0}\cdots
j_{n}}^{\ast}-\dim_{k}\left(  \left[  X\right]  _{j_{0}\cdots j_{n}}^{\ast
}\cap k^{n-m}\right)  \text{,} \label{irrqstar}%
\end{align}
where $\left[  X\right]  _{j_{0}\cdots j_{n}}$ denotes the linear span over
$k$ of the subset
\begin{equation}
\left\{  \left(  \frac{X_{j_{0}\cdots j_{s-1}j_{t}j_{s+1}\cdots j_{m}}%
}{X_{j_{0}\cdots j_{m}}}\right)  _{0\leq s\leq m}:m+1\leq t\leq n\right\}
\label{xjn}%
\end{equation}
of $K^{m+1}$, and $\left[  X\right]  _{j_{0}\cdots j_{n}}^{\ast}$ denotes the
linear span over $k$ of the subset
\[
\left\{  \left(  \frac{X_{j_{0}\cdots j_{t-1}j_{s}j_{t+1}\cdots j_{m}}%
}{X_{j_{0}\cdots j_{m}}}\right)  _{m+1\leq s\leq n}:0\leq t\leq m\right\}
\]
of $K^{n-m}$. Then our purpose is to show that $\operatorname{Irr}_{k}$ and
$\operatorname{Irr}_{k}^{\ast}$ are well-defined (that is, for every subspace
$X$ of $P^{n}(K)$, $\operatorname{Irr}_{k}X$ and $\operatorname{Irr}_{k}%
^{\ast}X$ are independent of the choice of $j_{0}\cdots j_{n}$) and that
$\operatorname{Irr}_{k}V$\ and $\operatorname{Irr}_{k}^{\ast}V$\ are the
answers to (i) and (ii), respectively, that is,
\begin{align}
\dim\bigvee_{\sigma\in\Gamma}\sigma(V)  &  =\dim V+\operatorname{Irr}%
_{k}V\text{,}\label{equality1}\\
\dim\bigcap_{\sigma\in\Gamma}\sigma(V)  &  =\dim V-\operatorname{Irr}%
_{k}^{\ast}V\text{,} \label{equality2}%
\end{align}
where $\bigvee$ denotes the join operation and $\bigcap$ denotes the meet operation.

In the following sections, we actually prove that $\operatorname{Irr}_{k}$ is
well-defined and (\ref{equality1}) holds and that $\operatorname{Irr}%
_{k}^{\ast}$ is well-defined and (\ref{equality2}) holds, revealing the
`duality' of these two.

\section{\label{duality}Duality of (\ref{equality1}) and (\ref{equality2})}

For each subspace $X$ of $P^{n}(K)$, let $X^{\perp}$ denote
\[
\left\{  (x_{i})_{0\leq i\leq n}\in P^{n}(K):\forall(a_{i})_{0\leq i\leq n}\in
X\text{,\ }\sum_{i=0}^{n}a_{i}x_{i}=0\right\}  \text{,}%
\]
which is a subspace of $P^{n}(K)$ of dimension $n-1-\dim X$ that is hereafter
called the \emph{dual} of $X$---though informally. Then, since the join of
subspaces of $P^{n}(K)$ is the dual of the meet of their duals, we have
\[
\bigvee_{\sigma\in\Gamma}\sigma(V)=\left(  \bigcap_{\sigma\in\Gamma}%
\sigma(V)^{\perp}\right)  ^{\perp}\text{,}%
\]
which is equivalent to
\begin{equation}
\bigvee_{\sigma\in\Gamma}\sigma(V)=\left(  \bigcap_{\sigma\in\Gamma}%
\sigma\left(  V^{\perp}\right)  \right)  ^{\perp} \label{dual}%
\end{equation}
because every $\sigma\in\Gamma$ satisfies
\begin{align*}
\sigma(V)^{\perp}  &  =\left\{  (x_{i})_{0\leq i\leq n}\in P^{n}%
(K):\forall(a_{i})_{0\leq i\leq n}\in\sigma(V)\text{,\ }\sum_{i=0}^{n}%
a_{i}x_{i}=0\right\} \\
&  =\left\{  (x_{i})_{0\leq i\leq n}\in P^{n}(K):\forall(a_{i})_{0\leq i\leq
n}\in V\text{,\ }\sum_{i=0}^{n}\sigma(a_{i})x_{i}=0\right\} \\
&  =\left\{  (x_{i})_{0\leq i\leq n}\in P^{n}(K):\forall(a_{i})_{0\leq i\leq
n}\in V\text{,\ }\sigma\left(  \sum_{i=0}^{n}a_{i}\sigma^{-1}(x_{i})\right)
=0\right\} \\
&  =\left\{  (x_{i})_{0\leq i\leq n}\in P^{n}(K):\forall(a_{i})_{0\leq i\leq
n}\in V\text{,\ }\sum_{i=0}^{n}a_{i}\sigma^{-1}(x_{i})=0\right\} \\
&  =\sigma\left(  V^{\perp}\right)  \text{.}%
\end{align*}
Also we have

\begin{proposition}
\label{irrirrdu}For every $m$-dimensional subspace $X$ of $P^{n}(K)$, for
every permutation $j_{0}\cdots j_{n}$ of $0\cdots n$ such that $X_{j_{0}\cdots
j_{m}}\neq0$, the right-hand side of \emph{(\ref{irrq})} is equal to the
expression obtained by replacing $X$ by $X^{\perp}$, $j_{0}\cdots j_{n}$ by
$j_{m+1}\cdots j_{n}j_{0}\cdots j_{m}$ and $n-m$ by $n-\dim X^{\perp
}=n-(n-1-m)=m+1$ in the right-hand side of \emph{(\ref{irrqstar})}, that is,
to
\[
\dim_{k}\left[  X^{\perp}\right]  _{j_{m+1}\cdots j_{n}j_{0}\cdots j_{m}%
}^{\ast}-\dim_{k}\left(  \left[  X^{\perp}\right]  _{j_{m+1}\cdots j_{n}%
j_{0}\cdots j_{m}}^{\ast}\cap k^{m+1}\right)  \text{,}%
\]
where $\left[  X^{\perp}\right]  _{j_{m+1}\cdots j_{n}j_{0}\cdots j_{m}}%
^{\ast}$ denotes the linear span over $k$ of the subset
\begin{equation}
\left\{  \left(  \frac{X_{j_{m+1}\cdots j_{t-1}j_{s}j_{t+1}\cdots j_{n}%
}^{\perp}}{X_{j_{m+1}\cdots j_{n}}^{\perp}}\right)  _{0\leq s\leq m}:m+1\leq
t\leq n\right\}  \label{xdualh}%
\end{equation}
of $K^{m+1}$.
\end{proposition}

\begin{proof}
By the well-known relation between the Pl\"{u}cker coordinates and the dual
Pl\"{u}cker coordinates of a space \cite[Chapter VII, \S \ 3, Theorem
I]{HodgePedoe}, letting $\epsilon$ and $\delta$ be the Levi-Civita symbol and
the generalized Kronecker delta symbol, respectively, we have
\begin{align*}
\frac{X_{j_{0}\cdots j_{s-1}j_{t}j_{s+1}\cdots j_{m}}}{X_{j_{0}\cdots j_{m}}}
&  =\frac{\epsilon_{j_{0}\cdots j_{s-1}j_{t}j_{s+1}\cdots j_{t-1}j_{s}%
j_{t+1}\cdots j_{n}}X_{j_{m+1}\cdots j_{t-1}j_{s}j_{t+1}\cdots j_{n}}^{\perp}%
}{\epsilon_{j_{0}\cdots j_{n}}X_{j_{m+1}\cdots j_{n}}^{\perp}}\\
&  =\delta_{j_{0}\cdots j_{s-1}j_{t}j_{s+1}\cdots j_{t-1}j_{s}j_{t+1}\cdots
j_{n}}^{j_{0}\cdots j_{n}}\frac{X_{j_{m+1}\cdots j_{t-1}j_{s}j_{t+1}\cdots
j_{n}}^{\perp}}{X_{j_{m+1}\cdots j_{n}}^{\perp}}\\
&  =-\frac{X_{j_{m+1}\cdots j_{t-1}j_{s}j_{t+1}\cdots j_{n}}^{\perp}%
}{X_{j_{m+1}\cdots j_{n}}^{\perp}}%
\end{align*}
for every $s$ with $0\leq s\leq m$ and every $t$ with $m+1\leq t\leq n$.
Therefore (\ref{xjn}) is equal to the subset
\[
\left\{  -\left(  \frac{X_{j_{m+1}\cdots j_{t-1}j_{s}j_{t+1}\cdots j_{n}%
}^{\perp}}{X_{j_{m+1}\cdots j_{n}}^{\perp}}\right)  _{0\leq s\leq m}:m+1\leq
t\leq n\right\}
\]
of $K^{m+1}$, that is, to the set of additive inverses of the elements of
(\ref{xdualh}), which implies
\[
\left[  X\right]  _{j_{0}\cdots j_{n}}=\left[  X^{\perp}\right]
_{j_{m+1}\cdots j_{n}j_{0}\cdots j_{m}}^{\ast}%
\]
and hence the desired equality.
\end{proof}

Proposition \ref{irrirrdu} is easily seen to imply that if one of
$\operatorname{Irr}_{k}$ and $\operatorname{Irr}_{k}^{\ast}$ is well-defined,
then the other is also well-defined and
\begin{equation}
\operatorname{Irr}_{k}V=\operatorname{Irr}_{k}^{\ast}V^{\perp}
\label{irrirrdual}%
\end{equation}
holds; (\ref{dual}) and (\ref{irrirrdual}) imply that (\ref{equality1}) is
equivalent to
\[
\dim\left(  \bigcap_{\sigma\in\Gamma}\sigma\left(  V^{\perp}\right)  \right)
^{\perp}=\dim V+\operatorname{Irr}_{k}^{\ast}V^{\perp}%
\]
and hence to
\[
\dim\bigcap_{\sigma\in\Gamma}\sigma\left(  V^{\perp}\right)  =\dim V^{\perp
}-\operatorname{Irr}_{k}^{\ast}V^{\perp}\text{,}%
\]
that is, to the equality obtained by replacing $V$ by $V^{\perp}$ in
(\ref{equality2}). Therefore, to prove that $\operatorname{Irr}_{k}$ is
well-defined and (\ref{equality1}) holds and that $\operatorname{Irr}%
_{k}^{\ast}$ is well-defined and (\ref{equality2}) holds, it is enough to
prove one of these two, say, the latter, which we prove in the next section.

\section{\label{prth}Proof of (\ref{equality2})}

Hereafter a subspace of $P^{n}(K)$ or $K^{n+1}$ is said to be $k$%
-\emph{rational }if it is spanned by a subset of $P^{n}(k)$ or $k^{n+1}$,
respectively, where (and hereafter) $P^{n}(k)$ denotes the image of
$k^{n+1}-\{\mathbf{0}\}$ by the canonical surjection $K^{n+1}-\{\boldsymbol{0}%
\}\rightarrow P^{n}(K)$. Now let $\Gamma$ act on $K^{n+1}$ by (\ref{action})
for all $\sigma\in\Gamma$ and $(x_{i})_{0\leq i\leq n}\in K^{n+1}$. Then, as
is seen---though more or less indirectly---from \cite[Chapter II, \S \ 8, no.
7, Theorem 1 (i)]{Bourbaki1} or its specialization \cite[Chapter V, \S \ 10,
no. 4, Proposition 6 \emph{a})]{Bourbaki2}, it holds that
\begin{align*}
&  \text{a subspace of }K^{n+1}\text{ is }k\text{-rational if and only if }\\
&  \text{it is stable under the action of }\Gamma\text{ on }K^{n+1}\text{,}%
\end{align*}
which is easily shown to imply that
\begin{align*}
&  \text{a subspace of }P^{n}(K)\text{ is }k\text{-rational if and only if }\\
&  \text{it is stable under the action of }\Gamma\text{ on }P^{n}(K)\text{,}%
\end{align*}
which implies that
\begin{equation}
\bigcap_{\sigma\in\Gamma}\sigma(V)\text{ is }k\text{-rational}
\label{krational}%
\end{equation}
because every $\tau\in\Gamma$ satisfies $\tau\Gamma=\Gamma$ and hence
\[
\tau\left(  \bigcap_{\sigma\in\Gamma}\sigma(V)\right)  =\bigcap_{\sigma
\in\Gamma}\tau\left(  \sigma(V)\right)  =\bigcap_{\sigma\in\Gamma}\tau
\sigma(V)=\bigcap_{\sigma\in\tau\Gamma}\sigma(V)=\bigcap_{\sigma\in\Gamma
}\sigma(V)\text{.}%
\]

For each subspace $X$ of $P^{n}(K)$, let $\widetilde{X}$ denote the span of
$X\cap P^{n}(k)$ in $P^{n}(K)$, which is the largest $k$-rational subspace of
$P^{n}(K)$ contained in $X$. Then, since a subspace of $P^{n}(K)$ is
$k$-rational if and only if it is the largest $k$-rational subspace of
$P^{n}(K)$ contained in itself, (\ref{krational}) is equivalent to
\[
\bigcap_{\sigma\in\Gamma}\sigma(V)=\widetilde{\bigcap_{\sigma\in\Gamma}%
\sigma(V)}\text{,}%
\]
which is equivalent to
\begin{equation}
\bigcap_{\sigma\in\Gamma}\sigma(V)=\widetilde{V} \label{vchilda}%
\end{equation}
because
\begin{align*}
\left(  \bigcap_{\sigma\in\Gamma}\sigma(V)\right)  \cap P^{n}(k)  &
=\bigcap_{\sigma\in\Gamma}\left(  \sigma(V)\cap P^{n}(k)\right)
=\bigcap_{\sigma\in\Gamma}\left(  \sigma(V)\cap\sigma\left(  P^{n}(k)\right)
\right) \\
&  =\bigcap_{\sigma\in\Gamma}\sigma\left(  V\cap P^{n}(k)\right)
=\bigcap_{\sigma\in\Gamma}\left(  V\cap P^{n}(k)\right)  =V\cap P^{n}%
(k)\text{.}%
\end{align*}
Also we have

\begin{proposition}
\label{dimvpai}For every $m$-dimensional subspace $X$ of $P^{n}(K)$, for every
permutation $j_{0}\cdots j_{n}$ of $0\cdots n$ such that $X_{j_{0}\cdots
j_{m}}\neq0$, the right-hand side of \emph{(\ref{irrqstar})} is equal to
$m-\dim\widetilde{X}$.
\end{proposition}

\begin{proof}
Let $\mu$ be the $k$-linear map $k^{n+1}\rightarrow K^{n-m}$\ defined by
\begin{align*}
\mu\left(  (x_{i})_{0\leq i\leq n}\right)   &  =\left(  x_{j_{s}}-\sum
_{t=0}^{m}x_{j_{t}}\frac{X_{j_{0}\cdots j_{t-1}j_{s}j_{t+1}\cdots j_{m}}%
}{X_{j_{0}\cdots j_{m}}}\right)  _{m+1\leq s\leq n}\\
&  =\left(  x_{j_{s}}\right)  _{m+1\leq s\leq n}-\sum_{t=0}^{m}x_{j_{t}%
}\left(  \frac{X_{j_{0}\cdots j_{t-1}j_{s}j_{t+1}\cdots j_{m}}}{X_{j_{0}\cdots
j_{m}}}\right)  _{m+1\leq s\leq n}\text{.}%
\end{align*}
Then we have
\[
\operatorname{Im}\mu=k^{n-m}+\left[  X\right]  _{j_{0}\cdots j_{n}}^{\ast}%
\]
and hence
\begin{align*}
\dim_{k}\operatorname{Ker}\mu &  =n+1-\dim_{k}\operatorname{Im}\mu\\
&  =n+1-\dim_{k}\left(  k^{n-m}+\left[  X\right]  _{j_{0}\cdots j_{n}}^{\ast
}\right) \\
&  =n+1-\left\{  n-m+\dim_{k}\left[  X\right]  _{j_{0}\cdots j_{n}}^{\ast
}-\dim_{k}\left(  k^{n-m}\cap\left[  X\right]  _{j_{0}\cdots j_{n}}^{\ast
}\right)  \right\} \\
&  =m+1-\left\{  \dim_{k}\left[  X\right]  _{j_{0}\cdots j_{n}}^{\ast}%
-\dim_{k}\left(  \left[  X\right]  _{j_{0}\cdots j_{n}}^{\ast}\cap
k^{n-m}\right)  \right\}  \text{.}%
\end{align*}
Let $X^{\prime}$ and $\widetilde{X}^{\prime}$ denote the subspaces of
$K^{n+1}$ such that $X^{\prime}-\{\boldsymbol{0}\}$ and $\widetilde{X}%
^{\prime}-\{\boldsymbol{0}\}$ are mapped by the canonical surjection
$K^{n+1}-\{\boldsymbol{0}\}\rightarrow P^{n}(K)$ onto $X$ and $\widetilde{X}$,
respectively. Then we can easily show that $\widetilde{X}^{\prime}$ is the
span of $X^{\prime}\cap k^{n+1}$ in $K^{n+1}$ and hence
\[
\dim_{k}\left(  X^{\prime}\cap k^{n+1}\right)  =\dim\widetilde{X}^{\prime
}=\dim\widetilde{X}+1=m+1-\left(  m-\dim\widetilde{X}\right)  \text{.}%
\]
Therefore, for the above $\mu$ and $X^{\prime}$, the desired
\[
\dim_{k}\left[  X\right]  _{j_{0}\cdots j_{n}}^{\ast}-\dim_{k}\left(  \left[
X\right]  _{j_{0}\cdots j_{n}}^{\ast}\cap k^{n-m}\right)  =m-\dim\widetilde{X}%
\]
is equivalent to
\[
\dim_{k}\operatorname{Ker}\mu=\dim_{k}\left(  X^{\prime}\cap k^{n+1}\right)
\]
and hence is implied by
\[
\operatorname{Ker}\mu=X^{\prime}\cap k^{n+1}\text{,}%
\]
which certainly holds since
\[
x_{j_{s}}-\sum_{t=0}^{m}x_{j_{t}}\frac{X_{j_{0}\cdots j_{t-1}j_{s}%
j_{t+1}\cdots j_{m}}}{X_{j_{0}\cdots j_{m}}}=0\text{ \ (}s=m+1,\ldots
,n\text{)}%
\]
is, as is virtually proved in \cite[Chapter VII, \S \ 2, the third and fourth
paragraphs]{HodgePedoe}, a necessary and sufficient condition for
$(x_{i})_{0\leq i\leq n}\in P^{n}(K)$ to be in $X$, and hence for
$(x_{i})_{0\leq i\leq n}\in K^{n+1}$ to be in $X^{\prime}$.
\end{proof}

Proposition \ref{dimvpai} is immediately seen to imply that
$\operatorname{Irr}_{k}^{\ast}$ is well-defined and
\begin{equation}
\dim\widetilde{V}=\dim V-\operatorname{Irr}_{k}^{\ast}V \label{vchildadim}%
\end{equation}
holds; (\ref{vchilda}) and (\ref{vchildadim}) imply (\ref{equality2}).
Therefore we have the desired result.

\begin{remark}
For each subspace $X$ of $P^{n}(K)$, let $\overline{X}$ denote $\widetilde
{X^{\perp}}^{\perp}$, which is the dual of the largest $k$-rational subspace
of $P^{n}(K)$ contained in $X^{\perp}$, that is, is the smallest $k$-rational
subspace of $P^{n}(K)$ containing $X$. Then we have, dually to (\ref{vchilda})
and (\ref{vchildadim}),
\begin{equation}
\bigvee_{\sigma\in\Gamma}\sigma(V)=\overline{V}\text{ and }\dim\overline
{V}=\dim V+\operatorname{Irr}_{k}V\text{,} \label{joinver}%
\end{equation}
which, by (\ref{dual}) and (\ref{irrirrdual}), are respectively equivalent to
\[
\left(  \bigcap_{\sigma\in\Gamma}\sigma\left(  V^{\perp}\right)  \right)
^{\perp}=\widetilde{V^{\perp}}^{\perp}\text{ and }\dim\widetilde{V^{\perp}%
}^{\perp}=\dim V+\operatorname{Irr}_{k}^{\ast}V^{\perp}%
\]
and hence to
\[
\bigcap_{\sigma\in\Gamma}\sigma\left(  V^{\perp}\right)  =\widetilde{V^{\perp
}}\text{ and }\dim\widetilde{V^{\perp}}=\dim V^{\perp}-\operatorname{Irr}%
_{k}^{\ast}V^{\perp}\text{,}%
\]
that is, to the equalities obtained by replacing $V$ by $V^{\perp}$ in
(\ref{vchilda}) and (\ref{vchildadim}).
\end{remark}

\begin{remark}
Since $V$ is $k$-rational if and only if $\widetilde{V}$ (or $\overline{V}$)
coincides with $V$, and hence if and only if $\widetilde{V}$ (or $\overline
{V}$) has the same dimension as $V$, (\ref{vchildadim}) and the latter of
(\ref{joinver}) make it natural to regard $\operatorname{Irr}_{k}^{\ast}V$ and
$\operatorname{Irr}_{k}V$ as `$k$-irrationality' measures of $V$.
\end{remark}

\begin{remark}
We can easily see that every $\sigma\in\Gamma$ satisfies $\widetilde
{\sigma(V)}=\widetilde{V}$ and $\overline{\sigma(V)}=\overline{V}$ as well as
$\dim\sigma(V)=\dim V$, which, by (\ref{vchildadim}) and the latter of
(\ref{joinver}), implies that $\operatorname{Irr}_{k}^{\ast}V$ and
$\operatorname{Irr}_{k}V$ are invariants of $V$\ under $\Gamma$.
\end{remark}

\begin{remark}
The part `every $\tau\in\Gamma$ satisfies \ldots' of the third sentence of
this section, which is valid only under our assumption that $\Gamma$\emph{ is
a subgroup of }$\operatorname{Aut}(K)$, can be replaced by `every $\tau
\in\Gamma$ satisfies $\tau\Gamma\subseteq\Gamma$ and hence
\[
\tau\left(  \bigcap_{\sigma\in\Gamma}\sigma(V)\right)  =\bigcap_{\sigma
\in\Gamma}\tau\left(  \sigma(V)\right)  =\bigcap_{\sigma\in\Gamma}\tau
\sigma(V)=\bigcap_{\sigma\in\tau\Gamma}\sigma(V)\supseteq\bigcap_{\sigma
\in\Gamma}\sigma(V)\text{,}%
\]
which implies
\[
\dim\tau\left(  \bigcap_{\sigma\in\Gamma}\sigma(V)\right)  =\dim
\bigcap_{\sigma\in\Gamma}\sigma(V)\implies\tau\left(  \bigcap_{\sigma\in
\Gamma}\sigma(V)\right)  =\bigcap_{\sigma\in\Gamma}\sigma(V)
\]
and hence
\[
\tau\left(  \bigcap_{\sigma\in\Gamma}\sigma(V)\right)  =\bigcap_{\sigma
\in\Gamma}\sigma(V)\text{',}%
\]
which, as well as all of Section \ref{duality} and this section except this
part, is valid under the weaker assumption that $\Gamma$\emph{ is a
subsemigroup of }$\operatorname{Aut}(K)$. Therefore our results hold under
this weaker assumption, which is equivalent to the original assumption when
$\Gamma$ is finite.
\end{remark}

\end{document}